\documentclass[DIV=calc, paper=letter, fontsize=11pt]{scrartcl}	 

\usepackage[]{units}
\usepackage{lipsum} 
\usepackage[english]{babel} 
\usepackage[protrusion=true,expansion=true]{microtype} 
\usepackage{amsmath,amssymb,amsfonts,amsthm} 

\usepackage[svgnames, table]{xcolor} 
\usepackage[hang, small,labelfont=bf,up,textfont=it,up]{caption} 
\usepackage{booktabs} 
\usepackage{fix-cm}	 
\usepackage[]{units}
\usepackage{rotating}
\usepackage{tabularx} 
\usepackage{setspace}
\usepackage[colorlinks=true, linkcolor=blue, citecolor=blue, filecolor=blue, runcolor=blue, urlcolor = blue]{hyperref}

\usepackage{fullpage}
\usepackage{sectsty} 

\usepackage{fancyhdr} 
\usepackage{lastpage} 

\usepackage{nicefrac}
\usepackage{cite}

\usepackage{tikz}
\usetikzlibrary{calc,patterns,decorations.pathmorphing,decorations.markings}
\usepackage{pgfplots}

\usepackage{multicol}
\usepackage{float} 

\newcounter{tempcolnum}
\makeatletter
\newcommand{\multicolinterrupt}[1]{
\setcounter{tempcolnum}{\col@number}
\end{multicols}
#1%
\begin{multicols}{\value{tempcolnum}}
}
\makeatother

\theoremstyle{definition}
\newtheorem{definition}{Definition}[section]

\theoremstyle{conjecture}

\theoremstyle{theorem}
\newtheorem{theorem}{Theorem}[section]

\theoremstyle{proposition}
\newtheorem{proposition}{Proposition}[section]

\theoremstyle{lemma}

\theoremstyle{corollary}
\newtheorem{corollary}{Corollary}[section]

\theoremstyle{remark}
\newtheorem{remark}{Remark}[section]

\RequirePackage[normalem]{ulem} 
\RequirePackage{color}\definecolor{RED}{rgb}{1,0,0}\definecolor{BLUE}{rgb}{0,0,1} 

\usepackage{graphicx}
\DeclareGraphicsExtensions{.png}

\usepackage{lettrine} 

\usepackage{caption}
\usepackage{graphicx}

\usepackage{titling} 

\newcommand{\HorRule}{\color{DarkGoldenrod} \rule{\linewidth}{1pt}} 

\pretitle{\vspace{-80pt} \begin{flushleft} \HorRule \fontsize{18}{18} \color{DarkRed} \selectfont} 

\title{\LARGE Vassiliev measures of complexity for open and closed curves in 3-space}

\posttitle{\par\end{flushleft}} 
\preauthor{\vspace{-5pt} \begin{flushleft}\fontsize{14}{14} \large  \color{DarkRed}} 
\author{Eleni Panagiotou$^{\#,*}$ and Louis H. Kauffman$^{\S}$ \\} 
\postauthor{\fontsize{12}{12} \small \color{Black} 
$^{\#}$Department of Mathematics and SimCenter, University of Tennessee at Chattanooga, TN 37403, USA, $<$eleni-panagiotou@utc.edu$>$  \\
$^{\S}$Department of Mathematics, Statistics and Computer Science
University of Illinois at Chicago
Chicago, IL 60607-7045, USA
and
Department of Mechanics and Mathematics
Novosibirsk State University
Novosibirsk
Russia
$<$kauffman@uic.edu$>$ \\
\lfoot{\footnotesize * Work supported by NSF DMS - 1913180}
\end{flushleft} \vspace{-18pt} \HorRule} 

\definecolor{issuePJA_color}{rgb}{1.0,0.0,0.0}

\definecolor{commentPJA_color}{rgb}{1.0,0.0,0.8}

\definecolor{commentEP_color}{rgb}{1.0,0.0,0.8}

\DeclareGraphicsExtensions{.png}

\begin{document}

\maketitle 

\thispagestyle{fancy} 

\textbf{In this manuscript we define Vassiliev measures of complexity for open curves in 3-space. These are related to the coefficients of the enhanced Jones polynomial of open curves in 3-space. These Vassiliev measures are continuous functions of the curve coordinates and as the ends of the curve tend to coincide, they converge to the corresponding Vassiliev invariants of the resulting knot. We focus on the second Vassiliev measure from the enhanced Jones polynomial for closed and open curves in 3-space. For closed curves, this second Vassiliev measure can be computed by a Gauss code diagram and it has an integral formulation, the double alternating self-linking integral. The double alternating self-linking integral is a topological invariant for closed curves and a continuous function of the curve coordinates for open curves in 3-space. For polygonal curves, the double alternating self-linking integral obtains a simpler expression in terms of geometric probabilities. For a polygonal curve with 4 edges, the double alternating self-linking integral coincides with the signed geometric probability of obtaining the knotoid k2.1 in a random projection direction.}

\noindent\textit{keywords} open knots, links, Vassiliev invariants, Gauss map

\section{Introduction}





Many physical systems are composed of entangled filamentous structures whose complexity affects their mechanical properties and their function \cite{Rubinstein2003,Everaers1996,Edwards1967,Edwards1968,Halverson2014,Grosberg1998,Grosberg1994}. Under some conditions, we can see these filamentous structures like mathematical curves in 3-space whose entanglement we can measure using tools from Knot Theory \cite{Kauffman2001,Grosberg1996,Grosberg2000,Flapan2019,Stolz2017,Sumners1988,Sumners1990,Arsuaga2005,Arsuaga2007,Marenduzzo2009,Micheletti2006,Diao1993,Millett2004}. A knot (link) is one (or more) simple closed curve(s) in 3-space and knots (links) are classified using the notion of topological equivalence. Many sophisticated topological invariants exist, such as knot and link polynomials \cite{Jones1985,Jones1987,Kauffman1987,Kauffman1990,Freyd1985}. However, there are two major throwbacks in measuring entanglement complexity of filamentous structures in practice: these may be open curves in 3-space (ie. they have a distinct starting and endpoint) and entanglement in these systems is very complex, at least in terms of number of crossings in diagrams, making the calculation of such polynomials intractable. 

The only measure of entanglement of open curves in 3-space until recently was the Gauss linking integral \cite{Gauss1877}. It measures self or pairwise entanglement of open curves and has had a lot of success across disciplines \cite{Panagiotou2010,Panagiotou2011,Panagiotou2013,Panagiotou2013b,Panagiotou2014,Panagiotou2015,Panagiotou2019,Klenin2000,Arsuaga2007a,Diao2005,Panagiotou2020,Baiesi2016,Baiesi2017,Baiesi2019}. Characterizing entanglement of open curves in 3-space using stronger measures of entanglement that can detect knotting of open curves has attracted a lot of attention in the last 20 years (see \cite{Millett2005,Sulkowska2012,Goundaroulis2017,Goundaroulis2017b,Rawdon2008} and references therein). However, all these approaches focused on approximating the open curve in 3-space by a knot (a closed curve) or by a knotoid (a 2-dimensional diagram). In 2020, in \cite{Panagiotou2020b} the Jones polynomial of open curves in 3-space was introduced and it was shown that it is a polynomial with real coefficients that are continuous functions of the curve coordinates. Therein it was shown that the Jones polynomial of open curves in 3-space generalizes the conventional Jones polynomial. In other words, the conventional Jones polynomial expression is a special case of the Jones polynomial introduced in \cite{Panagiotou2020b}. 

The approach introduced in \cite{Panagiotou2020b} provided a framework which we can use to study entanglement of both open and closed curves. This paper focuses in deriving Vassiliev invariant type measures of entanglement of both open and closed curves. In addition, an integral formula of the second Vassiliev invariant measure is introduced that provides a way to compute such entanglement measures directly from the coordinates of an open curve in 3-space.

Vassiliev invariants are related to the coefficients of the Jones polynomial and can distinguish knots and links as the polynomials do \cite{Vassiliev2005,Vassiliev1990,BarNatan1995,Goussarov2000,Polyak2001}. Combinatorial expressions for calculating some Vassiliev invariants from knot diagrams exist \cite{Vassiliev2005,Goussarov2000,Polyak2001} and integral expressions for Vassiliev invariants also exist, however their calculation remains elusive \cite{Hirshfeld1995,Thurston1995}. A major issue with computing Vassiliev invariants in practice is that physical filaments are usually composed by open curves in 3-space. The theory of knotoids provides a way to study complexity of open knot \textit{diagrams}, for which Vassiliev invariants are rigorously defined \cite{Turaev2012,Gugumcu2017,Gugumcu2017b,Manouras2020}. However, these are not well defined for open curves in 3-space. 





In this manuscript, we define Vassiliev measures for open curves in 3-space using the coefficients of the Jones polynomial with enhanced states of the open curves in 3-space. We show that they are continuous functions of the curve coordinates. An integral formula for the second Vassiliev invariant from the enhanced Jones polynomial is introduced, which involves a Gauss map, the double alternating self-linking integral. For polygonal curves in 3-space, the double alternating self-linking integral is expressed as a finite sum of geometric probabilities. For open curves, the double alternating self-linking integral is a continuous function of the curve coordinates. The double alternating self-linking integral that we introduce provides a unique - to our knowledge - measure of conformational complexity of open curves in 3-space that is stronger than the Gauss self-linking integral and does not require the computation of any knot polynomial. This can be extremely helpful in practice when studying entanglement in physical systems and we are aware of the potential breakthrough this could have in the study of proteins and polymers. We point out that the method introduced here can be applied to other Vassiliev invariants as well. This generates well defined integrals over closed or open curves in 3-space that capture higher degrees of entanglement.  



More precisely, in Section \ref{Vassiliev} we derive the exact formulas of the Vassiliev invariants for knots and knotoids obtained from the enhanced Jones polynomial. In Section \ref{Vasopen} we introduce the Vassiliev measures for open curves in 3-space and study their properties. In Section \ref{knots2} we focus on the second Vassiliev invariant of the enhanced Jones polynomial and show that it can be computed using a Gauss code. In Section \ref{knotoids2} a skein relation satisfied by the second Vassiliev invariant of the enhanced Jones polynomial for knotoids is derived and, in the case of knot-type knotoids, it is shown that this second Vassiliev invariant can be calculated by using a Gauss code diagram. In Section \ref{SLL} we introduce the double alternating self-linking integral and we show that it is equal (up to some specified constants) to the second Vassiliev invariant of the enhanced Jones polynomial in case of closed curves. In case of open curves, it is a continuous function of the curve coordinates. Finally, in Section \ref{VasPol} we show that for polygonal curves, the double alternating self-linking integral has a simpler expression as a finite sum of geometric probabilities and we give a finite form for the computation of those for polygonal curves with 3 or 4 edges in 3-space.

\section{Vassiliev invariants of knots and knotoids}\label{Vassiliev}

In this section we present the definitions of Vassiliev invariants of knots and knotoids defined through the coefficients of the Jones polynomial of knots and knotoids, repsectively.


\subsection{Vassiliev invariants of knots}

The Jones polynomial is defined using the normalized bracket polynomial. An expression of the Jones polynomial, which is helpful for deriving Vassiliev invariants, relies on using enhanced states. 

The bracket polynomial can be computed using the following relation:

\begin{equation}\label{skein}
\langle\raisebox{-10pt}{\includegraphics[width=.05\linewidth]{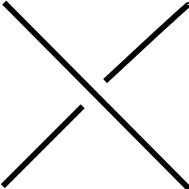}}\rangle=\langle\raisebox{-10pt}{\includegraphics[width=.05\linewidth]{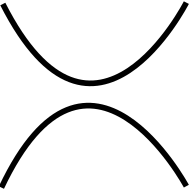}}\rangle-q\langle\raisebox{-10pt}{\includegraphics[width=.05\linewidth]{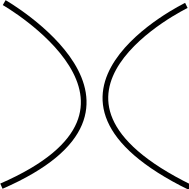}}\rangle,\hspace{0.5cm}\langle \bigcirc\rangle=q+q^{-1}.
\end{equation}

The smoothings with coefficient $1$ in the above equation are called $A$ smoothings, and the smoothings with coefficient $q$ are called $B$ smoothings. Through $A$ and $B$ smoothings of the crossings in a diagram, we obtain a set of enhanced states. In enhanced states, circles have an associated sign in their diagram, and a circle with no sign is a sum of two such states. A circle with a positive sign corresponds to $q$ and a circle with negative sign corresponds to $q^{-1}$.






For $s$ an enhanced state, let $i(s)$ be the number of B-smoothings and $\lambda(s)=$number of positive circles - number of negative circles. Then each state corresponds to a term of the form $(-q)^{i(s)}q^{\lambda(s)}=(-1)^{i(s)}q^{j(s)}$, where $j(s)=i(s)+\lambda(s)$, and the bracket polynomial can be expressed as

\begin{equation}
        \langle K\rangle=\sum_{s\in S}(-1)^{i(s)}q^{j(s)}.
\end{equation}

\noindent We define the enhanced Jones polynomial in $q$ as

\begin{equation}
        J_K(q)=q^{n_{+}-2n_{-}}(-1)^{n_-}\langle K\rangle,
\end{equation}

\noindent where $n_{+}$ is the number of positive crossings and $n_{-}$ is the number of negative crossings. The enhanced Jones polynomial, $J_K$, is related to the Jones polynomial, $V_K$, through the relation:

\begin{equation}\label{Jenhanced}
J_K(q)=(q+q^{-1})V_K(q),
\end{equation} 

\noindent with the substitution $q=t^{-1/2}$. 

The enhanced Jones polynomial satisfies the following relation:

\begin{equation}\label{Jskein}
q^{-2}J_{K_+}(q)-q^2J_{K_-}(q)=(q^{-1}-q)J_{K_0}.
\end{equation}





We use the substitution $t=e^x$, ie. $q=e^{-x/2}$, in $J_K(q)$ and then expand it in series of $x$.


\begin{equation}
    \begin{split}
        &J_K=\sum_{s\in S}(-1)^{i(s)+n_{-}}q^{j(s)+n_{+}-2n_{-}}\\
        &=\sum_{s\in S}(-1)^{i(s)+n_{-}}e^{-(j(s)+n_{+}-2n_{-})x/2}\\
        &=\sum_{s\in S}(-1)^{i(s)+n_{-}}\sum_{k=0}^{\infty}\frac{1}{k!2^k}((j(s)+n_{+}-2n_{-})x)^k\\
        &=\sum_{k=0}^{\infty}\frac{(-1)^kx^k}{2^kk!}\sum_{s\in S}(-1)^{i(s)+n_{-}}((j(s)+n_{+}-2n_{-})^k\\
    \end{split}
\end{equation}

The coefficients of $x^k$ are Vassiliev invariants of order $k$.

\begin{equation}\label{vk}
    v_k=\frac{(-1)^k}{2^kk!}\sum_{s\in S}(-1)^{i(s)+n_{-}}(j(s)+n_{+}-2n_{-})^k,
\end{equation}

\noindent where $S$ is the set of enhanced states of the knot diagram, $i(s)$ is the number of B-smoothings, $j(s)=i(s)+\lambda(s)$, $\lambda(s)$ is the algebraic number of circles and $n_{-}$ (resp. $n_{+}$) is the number of negative (resp. positive) crossings in the diagram.

To verify that $v_k$ is a finite type invariant of order $k$, we notice that, after substitution of $q=e^{-x/2}$ in Eq. \ref{Jskein}, we obtain

\begin{equation}
    \begin{split}
        J_{K_+}(x)-J_{K_{-}}(x)=x(\text{some mess}).
    \end{split}
\end{equation}

For a singular knot $K$ with $k$ double points, $J_K(x)$ is divisible by $x^k$, thus the coefficient of $x^k$ does not vanish. However, if $K$ has $k+1$ double points, the coefficient of $J_K(x)$ is divisible by $x^{k+1}$, which suggests that the coefficient of $x^k$ vanishes.
Thus, the coefficient of $x^k$, $v_k$, is a finite type invariant of degree $k$.




\subsection{Vassiliev invariants of knotoids}\label{Vasknot}








We will use the enhanced states expression of the Jones polynomial. Note that in enhanced states of knotoids, circles and arcs have an associated sign and if a circle or an arc has no sign, it is a sum of two states. A circle or an arc with a positive sign corresponds to $q$ and a circle or an arc with a negative sign corresponds to $q^{-1}$.


The enhanced states expression of the Jones polynomial in $q$ of a knotoid $K$ is:

\begin{equation}
    \begin{split}
        J_K(q)=q^{n_{+}-2n_{-}}\sum_{s\in S}(-1)^{i(s)+n_{-}}q^{j(s)}.
    \end{split}
\end{equation}

We use the substitution $q=e^{-x/2}$ and then expand $J_K(q)$ in series of $x$:
        
\begin{equation}
    \begin{split}
        J_K(x)&=\sum_{k=0}^{\infty}\frac{(-1)^kx^k}{2^kk!}\sum_{s\in S}(-1)^{i(s)+n_{-}}((j(s)+n_{+}-2n_{-})^k.
    \end{split}
\end{equation}

The coefficient of $x^k$ in $J_K(x)$ is a Vassiliev invariant of order $k$ of the knotoid $K$:

\begin{equation}\label{defnknotoid}
        v_k(K)=\frac{(-1)^k}{2^k k!}\sum_{s\in S}(-1)^{i(s)+n_{-}}((j(s)+n_{+}-2n_{-})^k,
\end{equation}

\noindent where $S$ is the set of enhanced states of the knotoid $K$, $i(s)$ is the number of B-smoothings, $j(s)=i(s)+\lambda(s)$, $\lambda(s)$ is the algebraic number of circles and $n_{-}$ (resp. $n_{+}$) is the number of negative (resp. positive) crossings in the knotoid diagram.









        


\section{Vassiliev measures of open curves in 3-space}\label{Vasopen}


In this section we define a set of new measures of entanglement of open curves in 3-space that we call Vassiliev measures due to the similarity of their definition to Vassiliev invariants. However, these are not invariants for open curves and in order to avoid any confusion, we will denote them $w_k$ (instead of $v_k$). In the following, we will show that the same definition applies to both open and closed curves in 3-space. For this reason we give the definition in general for any curve in 3-space:

\begin{definition}\label{defnVasopen}
Let $l$ denote an open or closed curve in 3-space. We define the $k$-th Vassiliev measure as:

\begin{equation}
    \begin{split}
        &w_k=\frac{(-1)^k}{4\pi2^k k!}\int_{\vec{\xi}\in S^2}\sum_{s_{\vec{\xi}}\in S_{\vec{\xi}}}(-1)^{i(s_{\vec{\xi}})+n_{-,\vec{\xi}}}((j(s_{\vec{\xi}})+n_{+,\vec{\xi}}-2n_{-,\vec{\xi}})^k)dA,
    \end{split}
\end{equation}

\noindent \noindent where $S$ is the set of enhanced states of a projection of $l$, $l_{\vec{\xi}}$, $i(s)$ is the number of B-smoothings, $j(s)=i(s)+\lambda(s)$, $\lambda(s)$ is the algebraic number of circles and $n_{-}$ (resp. $n_{+}$) is the number of negative (resp. positive) crossings in the diagram and where the integral is over all vectors in $S^2$ except a set of measure zero (corresponding to non-generic projections).
\end{definition}

\begin{proposition}
Let $l$ denote an open or closed curve in 3-space.
The $k$-th Vassiliev measure, $w_k(l)$, is defined by the coefficients of the enhanced Jones polynomial of $l$.
\end{proposition}

\begin{proof}
Let $l$ denote a curve in 3-space. Let $(l)_{\vec{\xi}}$ denote the projection of $l$ on a plane with normal vector $\vec{\xi}$. Let $K((l)_{\vec{\xi}})$ denote the knotoid corresponding to $(l)_{\vec{\xi}}$.

The normalized bracket polynomial of $l$ was defined in \cite{Panagiotou2020b} as: 

\begin{equation}\label{avnk}
V_{K(l)}=\frac{1}{4\pi}\int_{\vec{\xi}\in S^2}(-A^3)^{-wr(K(l)_{\vec{\xi}})}\langle K((l)_{\vec{\xi}})\rangle dA,
\end{equation}

\noindent where the integral is over all vectors in $S^2$ except a set of measure zero (corresponding to non-generic projections).
The integrand is the Jones polynomial of the knotoid $K((l)_{\vec{\xi}})$.

Using the enhanced states expression of the Jones polynomial in this case, we get

\begin{equation}
    \begin{split}
        J_K(q)&=\frac{1}{4\pi}\int_{\vec{\xi}\in S^2}q^{n_{+,\vec{\xi}}-2n_{-,\vec{\xi}}}\sum_{s_{\vec{\xi}}\in S_{\vec{\xi}}}(-1)^{i(s_{\vec{\xi}})+n_{-,\vec{\xi}}}q^{j(s_{\vec{\xi}})}dA.
    \end{split}
\end{equation}

We use the substitution $q=e^{-x/2}$ and then expand this in series of $x$.
        
\begin{equation}
    \begin{split}
        J_K(q)&=\sum_{k=0}^{\infty}\frac{(-1)^kx^k}{2^kk!}\frac{1}{4\pi}\int_{\vec{\xi}\in S^2}\sum_{s_{\vec{\xi}}\in S_{\vec{\xi}}}(-1)^{i(s_{\vec{\xi}})+n_{-,\vec{\xi}}}((j(s_{\vec{\xi}})+n_{+,\vec{\xi}}-2n_{-,\vec{\xi}})^k)dA.
    \end{split}
\end{equation}

Therefore, $w_k$ is the coefficient of $x^k$ in the enhanced Jones polynomial of $l$.

\end{proof}

\begin{corollary}\label{avv2}
Let $l$ denote an open curve in 3-space. The $k$th Vassiliev measure of $l$ derived from the enhanced Jones polynomial of $l$ is the average of the Vassiliev invariant in a projection over all possible projection directions, namely 

\begin{equation}
    w_k(l)=\frac{1}{4\pi}\int_{\vec{\xi}\in S^2}v_k(l_{\vec{\xi}})dA
\end{equation}

\end{corollary}

\begin{proof}
It follows directly from Definition \ref{defnVasopen} and Eq. \ref{defnknotoid}.
\end{proof}

\begin{proposition}
If $l$ is a closed curve in 3-space, then the $k$th Vassiliev measure, $w_k(l)$, is equal to the $k$th Vassiliev invariant, $v_k(l)$, obtained from the enhanced Jones polynomial of $l$, ie.  $w_k(l)=v_k(l)$

\end{proposition}

\begin{proof}
If $l$ is a closed curve in 3-space, then $v_k(l)$ is a topological invariant, independent on the projection direction. Thus

\begin{equation}
    \begin{split}
        &w_k(l)=\frac{1}{4\pi}\int_{\vec{\xi\in S^2}}v_k(l_{\vec{\xi}})dA=v_k(l_{\vec{\xi}})=v_k(l)
    \end{split}
\end{equation}

\end{proof}

\begin{proposition}\label{discp}
Let $l$ denote an open curve in 3-space, then $w_k$ can be expressed as 

\begin{equation}
    w_k(l)=\sum_{K_i\in K(l)}p(K_i)v_k(K_i)
\end{equation}

\noindent where $K_i$ is a knotoid that appears in a projection of $l$ and $p(K_i)$ is the geometric probability that the projection of $l$ gives the knotoid $K_i$ and $K(l)$ is the set of possible knotoids that can result as a projection of $l$.

\end{proposition}

\begin{proof}
Any fixed curve in 3-space can give projections that result in only a finite number of knotoids, we denote $K_i$, where $i=1,\dotsc,n$. Then $p(K_i)=\frac{1}{4\pi}A_i$, where $A_i$ is the sum of two antipodal spherical areas that define normal vectors to planes where the projection of $l$ gives the knotoid $K_i$. Since $v_k(l_{\vec{\xi}})$ is constant in these areas, 

\begin{equation}
\begin{split}
    w_k(l)&=\frac{1}{4\pi}\int_{\vec{\xi}}v_k(l_{\vec{\xi}})dA\\
    &=\sum_{1\leq i\leq n}\frac{1}{4\pi}\int_{\vec{\xi}\in A_i}v_k(l_{\vec{\xi}})dA=\frac{1}{4\pi}\sum_{1\leq i\leq n}v_k(K_i)\int_{\vec{\xi}\in A_i}dA\\
    &=\frac{1}{4\pi}\sum_{1\leq i\leq n}v_k(K_i)A_i=\sum_{1\leq i\leq n}p(K_i)v_k(K_i).
    \end{split}
\end{equation}

\end{proof}

\begin{remark}
We can also write $w_k$ as follows:

\begin{equation}
    \begin{split}
        &w_k=\frac{(-1)^k}{2^kk!}\sum_{S,n_{+},n_{-}}p_{K,S,n_{-},n_{+}}\sum_{s\in S}(-1)^{i(s)+n_{-}}((j(s)+n_{+}-2n_{-})^k),
    \end{split}
\end{equation}

\noindent where $p_{K,S,n_{+},n_{-}}$ denotes the geometric probability that a projection of $K$ has $n_{+},n_{-}$ positive and negative crossings respectively, and gives the set of enhanced states $S$ and the sum is taken over all possible sets of states $S$ that can be generated by projections of $K$ with a given type of crossings $n_{+},n_{-}$.

\end{remark}

\begin{proposition}
Let $l$ denote an open curve in 3-space.
Then the $k$th Vassiliev measure of $l$, $w_k(l)$ is a continuous function of the curve coordinates.
\end{proposition}

\begin{proof}
Let us consider a polygonal curve of $n$ edges, $l_n$. Then, by Proposition \ref{discp}, 

\begin{equation}
    w_k(l_n)=\sum_{K_i\in K(l)}p(K_i)v_k(K_i),
\end{equation}

\noindent where $K_i$ is a knotoid that appears in a projection of $l$ and $p(K_i)$ is the geometric probability that the projection of $l$ gives the knotoid $K_i$ and $K(l)$ is the set of possible knotoids that can result as a projection of $l$.
In \cite{Panagiotou2020b} it was shown that $p(K_i)$ is a continuous function of the curve coordinates. Thus $w_k(l_n)$ is a continuous function of the coordinates of $l_n$.
As $n\rightarrow\infty$, the result follows for any curve $l$.

\end{proof}


\section{The second Vassiliev invariant of the enhanced Jones polynomial of knots}\label{knots2}

In this section we study the second Vassiliev invariant of knots derived by the enhanced Jones polynomial and show that it can be calculated using a Gauss code diagram from any knot diagram.


\begin{theorem}\label{vaskeinknots}
The first three Vassiliev invariants defined by  the enhanced Jones polynomial, satisfy the equations

\begin{equation}
\begin{split}
&v_0(K)=v_0(\bigcirc)=2\\
&v_0(L)=v_0(\bigcirc\bigcirc)=4\\
& v_1(K_+)=v_1(K_-)\\
&    v_2(K_{+})-v_2(K_{-})=-6lk(K_0)
    \end{split}
\end{equation}

\noindent where $lk$ denotes the linking number of $K_0$
\end{theorem}

\begin{proof}
We start with the skein relation of the enhanced Jones polynomial:

\begin{equation}\label{Jskein2}
q^{-2}J_{K_+}(q)-q^2J_{K_-}(q)=(q^{-1}-q)J_{K_0}.
\end{equation}

\noindent After expressing $q$ as  $q=e^{-x/2}=\sum_{k=0}^{\infty}\frac{(-1)^kx^k}{2^kk!}$, we get

\begin{equation}
    J_{(K_0)}=\sum_{k=0}^{\infty}\frac{(-1)^kx^k}{2^kk!}\sum_{s_0\in S_0}(-1)^{i(s_0)+n_{-}^{(0)}}((j(s_0)+n_{+}^{(0)}-2n_{-}^{(0)})^k)
\end{equation}

\begin{equation}
    J_{(K_+)}=\sum_{k=0}^{\infty}\frac{(-1)^kx^k}{2^kk!}\sum_{s_+\in S_+}(-1)^{i(s_+)+n_{-}^{(+)}}((j(s_+)+n_{+}^{(+)}-2n_{-}^{(+)})^k)
\end{equation}

\begin{equation}
    J_{(K_-)}=\sum_{k=0}^{\infty}\frac{(-1)^kx^k}{2^kk!}\sum_{s_-\in S_-}(-1)^{i(s_-)+n_{-}^{(-)}}((j(s_-)+n_{+}^{(-)}-2n_{-}^{(-)})^k)
\end{equation}

We replace these expressions in Eq. \ref{Jskein2} and set $q=e^{-x/2}=\sum_{l=0}^{\infty}\frac{(-1)^lx^l}{2^ll!}$ to get:

\begin{equation}\label{eqrl}
    \begin{split}
    &(q^{-1}-q)J_{(K_0)}(q)=q^{-2}J_{(K_{+})}(q)-q^2J_{(K_{-})}(q)\Leftrightarrow\\
        &\Bigl(\sum_{l=1,odd}^{\infty}\frac{x^l}{l!2^{l-1}}\Bigr)\sum_{k=0}^{\infty}\frac{(-1)^kx^k}{2^kk!}\sum_{s_0\in S_0}(-1)^{i(s_0)+n_{-}^{(0)}}((j(s_0)+n_{+}^{(0)}-2n_{-}^{(0)})^k\\
        &=\Bigl(\sum_{l=0}^{\infty}\frac{1}{l!}x^l\Bigr)\sum_{k=0}^{\infty}\frac{(-1)^kx^k}{2^kk!}\sum_{s_+\in S_+}(-1)^{i(s_+)+n_{-}^{(+)}}((j(s_+)+n_{+}^{(+)}-2n_{-}^{(+)})^k\\
        &-\Bigl(\sum_{l=0}^{\infty}\frac{(-1)^lx^l}{l!}\Bigr)\sum_{k=0}^{\infty}\frac{(-1)^kx^k}{2^kk!}\sum_{s_-\in S_-}(-1)^{i(s_-)+n_{-}^{(-)}}((j(s_-)+n_{+}^{(-)}-2n_{-}^{(-)})^k\\
        \end{split}
\end{equation}

The left hand side of the latter equation can be expressed as

\begin{equation}\label{eqrl1}
    \begin{split}
        &\Bigl(\sum_{l=1,odd}^{\infty}\frac{x^l}{l!2^{l-1}}\Bigr)\sum_{k=0}^{\infty}\frac{(-1)^kx^k}{2^kk!}\sum_{s_0\in S_0}(-1)^{i(s_0)+n_{-}^{(0)}}((j(s_0)+n_{+}^{(0)}-2n_{-}^{(0)})^k\\
        &=(x+\frac{x^3}{24}+\dotsc)\cdot(\sum_{s_0\in S_0}(-1)^{i(s_0)+n_{-}^{(0)}}-\frac{1}{2}x\sum_{s_0\in S_0}(-1)^{i(s_0)+n_{-}^{(0)}}(j(s_0)+n_{+}^{(0)}-2n_{-}^{(0)})+\dotsc)\\
        &=x\sum_{s_0\in S_0}(-1)^{i(s_0)+n_{-}^{(0)}}-\frac{1}{2}x^2\sum_{s_0\in S_0}(-1)^{i(s_0)+n_{-}^{(0)}}(j(s_0)+n_{+}^{(0)}-2n_{-}^{(0)})+\dotsc\\
    \end{split}
\end{equation}

\noindent where the remaining terms involve higher powers of $x$.

The first sum in the right hand side can be expressed as

\begin{equation}\label{eqrl2}
    \begin{split}
    &\Bigl(\sum_{l=0}^{\infty}\frac{1}{l!}x^l\Bigr)\sum_{k=0}^{\infty}\frac{x^k}{2^kk!}\sum_{s_+\in S_+}(-1)^{i(s_+)+n_{-}^{(+)}}((j(s_+)+n_{+}^{(+)}-2n_{-}^{(+)})^k\\
         &=(1+x+\frac{1}{2}x^2+\dotsc)(\sum_{s_+\in S_+}(-1)^{i(s_+)+n_{-}^{(+)}}\\
        &-\frac{1}{2}x\sum_{s_+\in S_+}(-1)^{i(s_+)+n_{-}^{(+)}}(j(s_+)+n_{+}^{(+)}-2n_{-}^{(+)})+\frac{x^2}{8}\sum_{s_+\in S_+}(-1)^{i(s_+)+n_{-}^{(+)}}((j(s_+)+n_{+}^{(+)}-2n_{-}^{(+)})^2+\dotsc)\\
        &=\sum_{s_+\in S_+}(-1)^{i(s_+)+n_{-}^{(+)}}\\
        &-\frac{1}{2}x\sum_{s_+\in S_+}(-1)^{i(s_+)+n_{-}^{(+)}}(j(s_+)+n_{+}^{(+)}-2n_{-}^{(+)})+\frac{x^2}{8}\sum_{s_+\in S_+}(-1)^{i(s_+)+n_{-}^{(+)}}((j(s_+)+n_{+}^{(+)}-2n_{-}^{(+)})^2)\\
        &+x\sum_{s_+\in S_+}(-1)^{i(s_+)+n_{-}^{(+)}}-\frac{1}{2}x^2\sum_{s_+\in S_+}(-1)^{i(s_+)+n_{-}^{(+)}}(j(s_+)+n_{+}^{(+)}-2n_{-}^{(+)})\\
        &+\frac{1}{2}x^2\sum_{s_+\in S_+}(-1)^{i(s_+)+n_{-}^{(+)}}+\dotsc\\
        &=\sum_{s_+\in S_+}(-1)^{i(s_+)+n_{-}^{(+)}}+x\Bigl(-\frac{1}{2}\sum_{s_+\in S_+}(-1)^{i(s_+)+n_{-}^{(+)}}(j(s_+)+n_{+}^{(+)}-2n_{-}^{(+)})+\sum_{s_+\in S_+}(-1)^{i(s_+)+n_{-}^{(+)}}\Bigr)\\
        &+x^2\Bigl(\frac{1}{8}\sum_{s_+\in S_+}(-1)^{i(s_+)+n_{-}^{(+)}}((j(s_+)+n_{+}^{(+)}-2n_{-}^{(+)})^2)-\frac{1}{2}\sum_{s_+\in S_+}(-1)^{i(s_+)+n_{-}^{(+)}}(j(s_+)+n_{+}^{(+)}-2n_{-}^{(+)})\\
        &+\frac{1}{2}\sum_{s_+\in S_+}(-1)^{i(s_+)+n_{-}^{(+)}}\Bigr)+\dotsc
    \end{split}
\end{equation}

The second term in the right hand side gives:

\begin{equation}\label{eqrl3}
    \begin{split}
    &-\Bigl(\sum_{l=0}^{\infty}\frac{(-1)^lx^l}{l!}\Bigr)\sum_{k=0}^{\infty}\frac{(-1)^kx^k}{2^kk!}\sum_{s_-\in S_-}(-1)^{i(s_-)+n_{-}^{(-)}}((j(s_-)+n_{+}^{(-)}-2n_{-}^{(-)})^k\\
         &=-(1-x+\frac{x^2}{2}+\dotsc)(\sum_{s_-\in S_-}(-1)^{i(s_-)+n_{-}^{(-)}}+\\
        &-\frac{1}{2}x\sum_{s_-\in S_-}(-1)^{i(s_-)+n_{-}^{(-)}}(j(s_-)+n_{+}^{(-)}-2n_{-}^{(-)})+\frac{x^2}{8}\sum_{s_-\in S_-}(-1)^{i(s_-)+n_{-}^{(-)}}((j(s_-)+n_{+}^{(-)}-2n_{-}^{(-)})^2+\dotsc\\
        &=-\sum_{s_-\in S_-}(-1)^{i(s_-)+n_{-}^{(-)}}+\\
        &+\frac{1}{2}x\sum_{s_-\in S_-}(-1)^{i(s_-)+n_{-}^{(-)}}(j(s_-)+n_{+}^{(-)}-2n_{-}^{(-)})-\frac{x^2}{8}\sum_{s_-\in S_-}(-1)^{i(s_-)+n_{-}^{(-)}}((j(s_-)+n_{+}^{(-)}-2n_{-}^{(-)})^2\\
        &+x\sum_{s_-\in S_-}(-1)^{i(s_-)+n_{-}^{(-)}}-\frac{1}{2}x^2\sum_{s_-\in S_-}(-1)^{i(s_-)+n_{-}^{(-)}}(j(s_-)+n_{+}^{(-)}-2n_{-}^{(-)})\\
        &-\frac{1}{2}x^2\sum_{s_-\in S_-}(-1)^{i(s_-)+n_{-}^{(-)}}+\dotsc\\
        &=-\sum_{s_-\in S_-}(-1)^{i(s_-)+n_{-}^{(-)}}+x\Bigl(\frac{1}{2}\sum_{s_-\in S_-}(-1)^{i(s_-)+n_{-}^{(-)}}(j(s_-)+n_{+}^{(-)}-2n_{-}^{(-)})+\sum_{s_-\in S_-}(-1)^{i(s_-)+n_{-}^{(-)}}\Bigr)\\
        &+x^2\Bigl(-\frac{1}{8}\sum_{s_-\in S_-}(-1)^{i(s_-)+n_{-}^{(-)}}((j(s_-)+n_{+}^{(-)}-2n_{-}^{(-)})^2-\frac{1}{2}\sum_{s_-\in S_-}(-1)^{i(s_-)+n_{-}^{(-)}}(j(s_-)+n_{+}^{(-)}-2n_{-}^{(-)})\\
        &-\frac{1}{2}\sum_{s_-\in S_-}(-1)^{i(s_-)+n_{-}^{(-)}}\Bigr)+\dotsc
    \end{split}
\end{equation}

By replacing Eq. \ref{eqrl1}, Eq. \ref{eqrl2} and Eq. \ref{eqrl3} in Eq. \ref{eqrl} and by equating the coefficients of $x^0$ in the left and right hand side of Eq.\ref{eqrl} we obtain the following relation

\begin{equation}\label{eq2}
\begin{split}
&0=\sum_{s_+\in S_+}(-1)^{i(s_+)+n_{-}^{(+)}}-\sum_{s_-\in S_-}(-1)^{i(s_-)+n_{-}^{(-)}}\Leftrightarrow \sum_{s_+\in S_+}(-1)^{i(s_+)+n_{-}^{(+)}}=\sum_{s_-\in S_-}(-1)^{i(s_-)+n_{-}^{(-)}}\\
    &\Leftrightarrow v_0(K_+)=v_0(K_-)
    \end{split}
\end{equation}

Thus, $v_0$ is invariant upon a crossing change. So, if $K$ is a knot, then we can change crossings so that we obtain the unknot, thus

\begin{equation}\label{eq2}
    v_0(K)=v_0(\bigcirc)=2
\end{equation}

Similarly, if $L$ is a link of two components, we can change crossings in order to get the trivial link of two components:

\begin{equation}\label{eq2}
    v_0(L)=v_0(\bigcirc \bigcirc)=4
\end{equation}




By equating the coefficients of $x$ in Eq. \ref{eqrl}, we obtain the following relation

\begin{equation}\label{Eq31}
\begin{split}
&\Bigl(\sum_{s_0\in S_0}(-1)^{i(s_0)+n_{-}^{(0)}}\Bigr)\\
&=\Bigl(-\frac{1}{2}\sum_{s_+\in S_+}(-1)^{i(s_+)+n_{-}^{(+)}}(j(s_+)+n_{+}^{(+)}-2n_{-}^{(+)})+\sum_{s_+\in S_+}(-1)^{i(s_+)+n_{-}^{(+)}}\Bigr)+\\
&+\Bigl(\frac{1}{2}\sum_{s_-\in S_-}(-1)^{i(s_-)+n_{-}^{(-)}}(j(s_-)+n_{+}^{(-)}-2n_{-}^{(-)})+\sum_{s_-\in S_-}(-1)^{i(s_-)+n_{-}^{(-)}}\Bigr)\\
&\Leftrightarrow -\frac{1}{2}\sum_{s_+\in S_+}(-1)^{i(s_+)+n_{-}^{(+)}}(j(s_+)+n_{+}^{(+)}-2n_{-}^{(+)})+\frac{1}{2}\sum_{s_-\in S_-}(-1)^{i(s_-)+n_{-}^{(-)}}(j(s_-)+n_{+}^{(-)}-2n_{-}^{(-)})\\
&=\sum_{s_0\in S_0}(-1)^{i(s_0)+n_{-}^{(0)}}-\sum_{s_+\in S_+}(-1)^{i(s_+)+n_{-}^{(+)}}-\sum_{s_-\in S_-}(-1)^{i(s_-)+n_{-}^{(-)}}\\
&\Leftrightarrow v_1(K_+)-v_1(K_-)=v_0(K_0)-v_0(K_+)-v_0(K_-)
\end{split}
\end{equation}

If $K_+$ is a knot, then Eq. \ref{Eq31} gives

\begin{equation}
    \begin{split}
        v_1(K_+)-v_1(K_-)=v_0(K_0)-v_0(K_+)-v_0(K_-)=4-2-2=0
    \end{split}
\end{equation}

Thus, in that case

\begin{equation}
    v_1(K_-)=v_1(K_+)
\end{equation}

Therefore, $v_1(K)$ does not change upon crossing changes. So,

\begin{equation}
    v_1(K)=v_1(\bigcirc)=0
\end{equation}

If $K_+$ is a link, then $K_-$ is also a link and $K_0$ is a knot. Then Eq. \ref{Eq31} gives 

\begin{equation}
    \begin{split}
        v_1(K_+)-v_1(K_-)&=v_0(K_0)-v_0(K_+)-v_0(K_-)=2-4-4=-6
    \end{split}
\end{equation}

So, 

\begin{equation}
    \begin{split}
        v_1(K_+)&=v_1(K_-)-6
    \end{split}
\end{equation}

Given a link $K$, we can get the trivial link by changing crossings. Suppose that we need to change $\lambda$ crossings. Then $\lambda=lk$, where $lk$ is the linking number of $K$, and

\begin{equation}\label{lkn}
    \begin{split}
        v_1(K_+)&=-6lk(K_0)
    \end{split}
\end{equation}

By equating the coefficients of $x^2$ in Eq. \ref{eqrl} we get

\begin{equation}
    \begin{split}
    &-\frac{1}{2}\sum_{s_0\in S_0}(-1)^{i(s_0)+n_{-}^{(0)}}(j(s_0)+n_{+}^{(0)}-2n_{-}^{(0)})\\
    &=\Bigr(\frac{1}{8}\sum_{s_+\in S_+}(-1)^{i(s_+)+n_{-}^{(+)}}((j(s_+)+n_{+}^{(+)}-2n_{-}^{(+)})^2-\frac{1}{2}\sum_{s_+\in S_+}(-1)^{i(s_+)+n_{-}^{(+)}}(j(s_+)+n_{+}^{(+)}-2n_{-}^{(+)})\\
        &+\frac{1}{2}\sum_{s_+\in S_+}(-1)^{i(s_+)+n_{-}^{(+)}}\Bigr)\\
        &+\Bigl(-\frac{1}{8}\sum_{s_-\in S_-}(-1)^{i(s_-)+n_{-}^{(-)}}((j(s_-)+n_{+}^{(-)}-2n_{-}^{(-)})^2\\
        &-\frac{1}{2}\sum_{s_-\in S_-}(-1)^{i(s_-)+n_{-}^{(-)}}(j(s_-)+n_{+}^{(-)}-2n_{-}^{(-)})-\frac{1}{2}\sum_{s_-\in S_-}(-1)^{i(s_-)+n_{-}^{(-)}}\Bigr)\\
        &\Leftrightarrow \frac{1}{8}\sum_{s_+\in S_+}(-1)^{i(s_+)+n_{-}^{(+)}}((j(s_+)+n_{+}^{(+)}-2n_{-}^{(+)})^2\\
        &-\frac{1}{8}\sum_{s_-\in S_-}(-1)^{i(s_-)+n_{-}^{(-)}}((j(s_-)+n_{+}^{(-)}-2n_{-}^{(-)})^2\\
        &=-\frac{1}{2}\sum_{s_0\in S_0}(-1)^{i(s_0)+n_{-}^{(0)}}(j(s_0)+n_{+}^{(0)}-2n_{-}^{(0)})+\frac{1}{2}\sum_{s_+\in S_+}(-1)^{i(s_+)+n_{-}^{(+)}}(j(s_+)+n_{+}^{(+)}-2n_{-}^{(+)})\\
        &+\frac{1}{2}\sum_{s_-\in S_-}(-1)^{i(s_-)+n_{-}^{(-)}}(j(s_-)+n_{+}^{(-)}-2n_{-}^{(-)})-\frac{1}{2}\sum_{s_+\in S_+}(-1)^{i(s_+)+n_{-}^{(+)}}+\frac{1}{2}\sum_{s_-\in S_-}(-1)^{i(s_-)+n_{-}^{(-)}}\\
        &\Leftrightarrow v_2(K_+)-v_2(K_-)=v_1(K_0)-v_1(K_+)-v_1(K_-)-\frac{1}{2}v_0(K_+)+\frac{1}{2}v_0(K_-)\\
        &\Leftrightarrow v_2(K_+)-v_2(K_-)=v_1(K_0)\\
    \end{split}
\end{equation}

Thus, if $K_+$ is a knot, using Eq. \ref{lkn}, we get

\begin{equation}
    \begin{split}
    v_2(K_+)-v_2(K_-)=-6lk(K_0)
    \end{split}
\end{equation}

\end{proof}

\begin{remark}
It is interesting to point out that a similar relation exists for the Casson invariant, the second Vassiliev invariant obtained from the Conway polynomial \cite{Chmutov2012}.
\end{remark}

In the following we will refer to Gauss diagrams.

\begin{definition}
A Gauss diagram is a way to describe a knot diagram. A Gauss diagram is the of a immersing circle with the preimages of each double point (associated with the knot diagram) connected with a chord. To incorporate the information on overpasses and underpasses, chords have an orientation from the over arc to the under arc. Gauss diagrams can have a base point and an orientation that matches a base point on a knot and the orientation of the knot. 
Given a knot $K$, we will denote $\langle \text{Gauss diagram 1 + Gauss diagram 2}, K\rangle$ the sum over all subdiagrams of $K$ isomorphic to either Gauss diagram 1 or Gauss diagram 2.

A Gauss diagram can be used to describe knotoid diagrams as well. In this case, the Gauss diagram has a starting and an endpoint matching those of the knotoid. 
\end{definition}

\begin{theorem}\label{knotv2}
Let $\hat{v}_2$ be defined as follows:

\begin{equation}
    \hat{v}_{2}(K)=\frac{1}{2}\sum_{j_1>j_2>j_3>j_4\in I_{\vec{\xi}}\prime}\epsilon(j_1,j_3)\epsilon(j_2,j_4),
\end{equation}

\noindent where $I\prime$ denotes the set of pairs of alternating crossings in the diagram of the knot $K$.

$\hat{v}_2$ is a second Vassiliev invariant of knotoids, and $v_2=\frac{1}{4}+6\hat{v}_2$, where $v_2$ denotes the second Vassiliev invariant from the enhanced Jones polynomial. 
\end{theorem}

\begin{proof}
The proof is similar to Theorem 1.A in \cite{Polyak2001}. 







Notice that in terms of Gauss diagrams of knots, 

\begin{equation}
    \hat{v}_{2}(K)=\frac{1}{2}\langle\raisebox{-15pt}{\includegraphics[width=.1\linewidth]{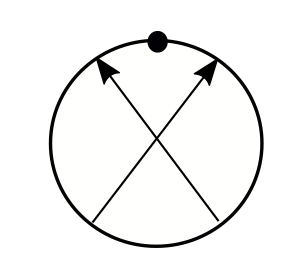}}+\raisebox{-15pt}{\includegraphics[width=.1\linewidth]{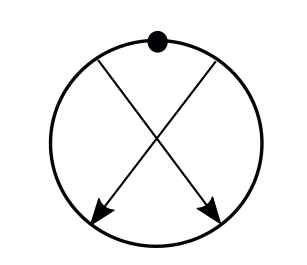}},K\rangle.
\end{equation}

To calculate $v_2$ of the knot $K$, we transform $K$ to
a descending knot diagram, going from the base point along the orientation of $K$ and replacing
an undercrossing by an overcrossing, if at the first passage through the point we
go along the undercrossing. When we pass over the whole diagram, it becomes
descending. Each time we change a crossing $s$, the value of $v_2$ changes by $(-6)(-\epsilon(s))lk(\raisebox{-10pt}{\includegraphics[width=.05\linewidth]{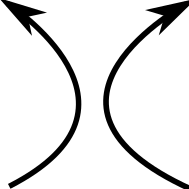}})$ where $\epsilon(s)$ is the sign of the crossing. Since
$v_2(descending)=\frac{1}{4}$, it gives

\begin{equation}\label{ucross}
    v_2(K)=\frac{1}{4}+6\sum_d\epsilon(s)lk(L_s),
\end{equation}

\noindent where $L_s$ runs over links which appeared as smoothings at points where the crossing changed to make $K$ a descending knot diagram.

To calculate $lk(L_s)$, we can sum up the signs of all the crossing points of $L_s$ in
which the component containing the base point goes below the other component.
These points correspond to chords of $G$ intersecting the chord $c(s)$ corresponding
to $s$ and directed to the side of $c(s)$ containing the base point. At the moment
all arrows of the original diagram G with heads between the base point and the
head of $c(s)$ have been inverted. Therefore $lk(L_s)$ is equal to the sum of signs of
arrows crossing $c(s)$ and having heads between tail of $c(s)$ and the base point. In
other words, $lk(L_s)$ is $\sum\epsilon(c_2)$ where the summation runs over all chords involved,
together with $c(s)$, into subdiagrams of the type $\raisebox{-10pt}{\includegraphics[width=.05\linewidth]{v2knots.png}}$.

On the other hand, we can obtain the same result by transforming $K$ to an ascending knot diagram, replacing an overcrossing by an undercrossing, if at first passage through the point we go along an overcrossing. Then again we obtain:

\begin{equation}\label{ucross2}
    v_2(K)=\frac{1}{4}+6\sum_a\epsilon(s)lk(L_s),
\end{equation}

\noindent where $L_s$ runs over links which appeared as smoothings at points where the crossing changed to make $K$ and ascending knot diagram.

To calculate $lk(L_s)$ in this case we can sum up all the crossing points of $L_s$ in which the component containing the base point goes over the other component. These points correspond to chords of $G$ intersecting the chord $c(s)$ corresponding to $s$ and directed to the side of $c(s)$ that does not contain the base point. At the moment all arrows of the original diagram $G$ with heads between the base point and the tail of $c(s)$ have been inverted. Therefore $lk(L_s)$ is equal to the sum of signs of arrows crossing $c(s)$ and having tails between the base point and the head of $c(s)$. In other words, $lk(L_s)$ is $\sum\epsilon(c(s))$, where the summation runs over all chords involved together with $c(s)$, into subdiagrams of the type $\raisebox{-10pt}{\includegraphics[width=.05\linewidth]{v2knots2.png}}$.

Thus, by Eq. \ref{ucross} and Eq. \ref{ucross2},

\begin{equation}\label{ucross3}
\begin{split}
    v_2(K)&=\frac{1}{2}(\frac{1}{4}+6\sum_a\epsilon(s)lk(L_s)+\frac{1}{4}+6\sum_d\epsilon(s)lk(L_s)\\
    &=\frac{1}{4}+6\hat{v}_2
    \end{split}
\end{equation}


\end{proof}

\section{The second Vassiliev invariant of knotoids}\label{knotoids2}

Here we study the second Vassiliev invariant 
of the enhanced Jones polynomial of knotoids.

\begin{definition}
We will define a separated linkoid diagram of two components, a linkoid diagram where one of the two components is all above or all below the other one.
\end{definition}



\begin{theorem}\label{vaskeinknotoids}
Let $K$ (resp. $L$) be a knotoid (resp. linkoid). Let $K_+,K_-,K_0$ be derived by changing and smoothing a positive crossing of $K$. Let $r$ be the algebraic sum of crossings needed to convert $K$ to an ascending knotoid diagram and let $l$ denote the algebraic sum of crossings needed to convert $K_0$ to a separated linkoid. Let $v_k(K_0^s)$ denote the $k$th Vassiliev invariant of the separated linkoid diagram obtained by $K_0$. The first three Vassiliev invariants defined by the coefficients of the enhanced Jones polynomial satisfy the equations

\begin{equation}
\begin{split}
&v_0(K)=v_0(\raisebox{-1pt}{\includegraphics[width=.05\linewidth]{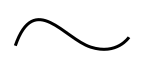}})=2\\
&v_0(L)=v_0(L^{s})\\
& v_1(K_+)-v_1(K_-)=rv_0(K_0^s)-4r\\
&    v_2(K_{+})-v_2(K_{-})=v_1(K_0^s)-(2l+2r+1)v_0(K_0^s)+2l+8r+4
    \end{split}
\end{equation}


\end{theorem}

\begin{proof}
We follow the same steps as in the proof of Theorem \ref{vaskeinknots}. Namely, by expanding the enhanced Jones polynomial of the knotoid $K$ in the skein relation of the enhanced Jones polynomial and equating the coefficients of $x^0$ we get the following relation for knotoids:

\begin{equation}\label{eq2}
\begin{split}
&0=\sum_{s_+\in S_+}(-1)^{i(s_+)+n_{-}^{(+)}}-\sum_{s_-\in S_-}(-1)^{i(s_-)+n_{-}^{(-)}}\Leftrightarrow \sum_{s_+\in S_+}(-1)^{i(s_+)+n_{-}^{(+)}}=\sum_{s_-\in S_-}(-1)^{i(s_-)+n_{-}^{(-)}}\\
    &\Leftrightarrow v_0(K_+)=v_0(K_-)
    \end{split}
\end{equation}

Thus, $v_0$ is invariant upon a crossing change. So, if $K$ is a knotoid, then we can change crossings so that we obtain an ascending diagram of a knotoid, which is the trivial knotoid \cite{Manouras2020}. Thus

\begin{equation}\label{eq20}
    v_0(K)=v_0(\bigcirc)=2
\end{equation}

Similarly, if $L$ is a linkoid of two components, we can change crossings in order to get a separated linkoid of two components:

\begin{equation}\label{eq2}
    v_0(L)=v_0(L^{s})
\end{equation}

By equating the coefficients of $x$ in the expansion of the enhanced Jones polynomial of knotoids, we obtain the following relation

\begin{equation}\label{Eq47}
\begin{split} v_1(K_+)-v_1(K_-)=v_0(K_0)-v_0(K_+)-v_0(K_-)
\end{split}
\end{equation}

If $K_+$ is a knotoid and $K_0$ is a linkoid of two components, then Eq. \ref{Eq47} and Eq. \ref{eq20} give

\begin{equation}
    \begin{split}
        v_1(K_+)-v_1(K_-)=v_0(K_0^{s})-4,
    \end{split}
\end{equation}

\noindent from which we can get the relations 

\begin{equation}
    \begin{split}
        &v_1(K_+)=v_1(K_-)+v_0(K_0^{s})-4\\
        &v_1(K_-)=v_1(K_+)-v_0(K_0^{s})+4\\
    \end{split}
\end{equation}

Thus, by repeatedly changing crossings in order to make $K_+$ into an ascending knotoid diagram (for which $v_1=0$), we get 

\begin{equation}
    \begin{split}
        v_1(K_+)=rv_0(K_0^{s})-4r,
    \end{split}
\end{equation}

\noindent where $r$ is the algebraic sum of the signs of the crossings that need to be changed in order to make the knotoid $K_+$ into an ascending knotoid diagram.





If $L_+$ is a linkoid, then $L_-$ is also a linkoid and $L_0$ is a knotoid. Then Eq. \ref{Eq47} gives 

\begin{equation}
    \begin{split}
        v_1(L_+)-v_1(L_-)&=v_0(L_0)-v_0(L_+)-v_0(L_-)=2-2v_0(L^{s}),
    \end{split}
\end{equation}

\noindent from which we get the following relations

\begin{equation}
    \begin{split}
        &v_1(L_+)=v_1(L_-)+2-2v_0(L^{s})\\
        &v_1(L_-)=v_1(L_+)-2+2v_0(L^{s})\\
    \end{split}
\end{equation}

We can convert $L_+$ to a separated linkoid by repeatedly changing crossings. Suppose that the algebraic sum of crossings we need to change in order to get a separated linkoid diagram, $L^{s}$, is $l$. Then we get

\begin{equation}
    \begin{split}
        v_1(L_+)&=v_1(L^{s})+2l-2lv_0(L^{s}).
    \end{split}
\end{equation}



By equating the coefficients of $x^2$ in the expansion of the enhanced Jones polynomial of knotoids, we get

\begin{equation}
    \begin{split}
     v_2(K_+)-v_2(K_-)=v_1(K_0)-v_1(K_+)-v_1(K_-)\\
    \end{split}
\end{equation}

Thus, if $K_+$ is a knotoid,

\begin{equation}
    \begin{split}
    v_2(K_+)-v_2(K_-)&=v_1(K_0^{s})+2l-2lv_0(K_0^{s})-[v_1(K_-)+v_0(K_0^{s})-4]-v_1(K_-)\\
    &=v_1(K_0^{s})+2l-2lv_0(K_0^{s})-v_0(K_0^{s})+4-2v_1(K_-)\\
    &=v_1(K_0^{s})-(2l+1)v_0(K_0^{s})+2l+4-2(rv_0(K_0^{s})-4r\\
    &=v_1(K_0^{s})-(2l+2r+1)v_0(K_0^{s})+2l+8r+4\\
    \end{split}
\end{equation}

\end{proof}

\begin{definition}
Let $L$ define a linkoid of two components. We define the linking number of $L$, we denote $lk(L)$, to be half the algebraic sum of inter-crossings in a diagram of $L$, ie. $lk(L)=\frac{1}{2}\sum_{c\in D}sign(c)$, where $D$ is the set of crossings between the two components in the diagram.
\end{definition}

\begin{remark}
Note that the linking number of linkoids is an invariant of linkoids. The linking number of knotoids is not an integer in general, but for a link-type linkoid, $lk$ is an integer.
\end{remark}

\begin{proposition}\label{vaskeinknotoidstype}
Let $K$, resp. $L$, denote a knot-type knotoid, a link-type linkoid, resp. The first three Vassiliev invariants defined by the coefficients of the enhanced Jones polynomial satisfy the equations

\begin{equation}
\begin{split}
&v_0(K)=v_0(\raisebox{-1pt}{\includegraphics[width=.05\linewidth]{trivial.png}})=2\\
&v_0(L)=v_0(\raisebox{-1pt}{\includegraphics[width=.05\linewidth]{trivial.png}}\raisebox{-1pt}{\includegraphics[width=.05\linewidth]{trivial.png}})=4\\
& v_1(K_+)=v_1(K_-)\\
&    v_2(K_{+})-v_2(K_{-})=-6lk
    \end{split}
\end{equation}

\noindent where $lk$ is the linking number of the linkoid $K_0$.

\end{proposition}

\begin{proof}
If $L^s$ is a separated link-type linkoid, it is the trivial linkoid, thus $v_0(L^s)=4$. The algebraic sum of crossings needed to change in order to get a separated linkoid diagram from any link-type linkoid diagram is $l=lk$, where $lk$ denotes the linking number of link-type linkoid. When $K$ is a knot-type knotoid, then $K_0$ is a link-type linkoid. Using these facts, the results follow using the same method as in the proof of Theorem \ref{vaskeinknotoids}.






\end{proof}

\begin{theorem}\label{knotoid2}
Let $K$ denote a knot-type knotoid. Let $\hat{v}_2$ be defined as follows:

\begin{equation}
    \hat{v}_{2}(K)=\frac{1}{2}\sum_{j_1>j_2>j_3>j_4\in I_{\vec{\xi}}\prime}\epsilon(j_1,j_3)\epsilon(j_2,j_4),
\end{equation}

\noindent where $I\prime$ denotes the set of pairs of alternating crossings in the diagram of the knotoid $K$.

$\hat{v}_2$ is a second Vassiliev invariant of knot-type knotoids, and $v_2=\frac{1}{4}+6\hat{v}_2$, where $v_2$ denotes the second Vassiliev invariant from the enhanced Jones polynomial. 
\end{theorem}

\begin{proof}
The proof follows a similar approach as the proof of Theorem \ref{knotv2}.







Notice that in terms of Gauss diagrams of knotoids, 

\begin{equation}
    \hat{v}_{2}(K)=\langle\raisebox{-15pt}{\includegraphics[width=.1\linewidth]{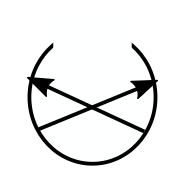}}+\raisebox{-15pt}{\includegraphics[width=.1\linewidth]{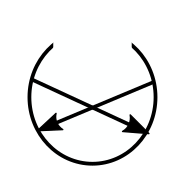}},K\rangle.
\end{equation}

To calculate $v_2$ of the knotoid $K$, we transform $K$ to
a descending knotoid, going from the starting point along the orientation of $K$ and replacing
an undercrossing by an overcrossing, if at the first passage through the point we
go along the undercrossing. When we arrive at the endpoint of the diagram, it becomes
descending. Each time we change a crossing $s$, the value of $v_2$ changes by $(-6)(-\epsilon(s))lk(\raisebox{-10pt}{\includegraphics[width=.05\linewidth]{k0.png}})$ where $\epsilon(s)$ is the sign of the crossing. Since
$v_2(descending)=\frac{1}{4}$, it gives

\begin{equation}\label{uknotoid}
    v_2(K)=\frac{1}{4}+6\sum\epsilon(s)lk(L_s),
\end{equation}

\noindent where $L_s$ runs over linkoids which appeared as smoothings at points where the crossing changed.

To calculate $lk(L_s)$, we can sum up the signs of all the crossing points of $L_s$ in
which the component containing the endpoints goes below the other component.
These points correspond to chords of $G$ intersecting the chord $c(s)$ corresponding
to $s$ and directed to the side of $c(s)$ containing the endpoints. At the moment
all arrows of the original diagram G with heads between the starting point and the
head of $c(s)$ have been inverted. Therefore $lk(L_s)$ is equal to the sum of signs of
arrows crossing $c(s)$ and having heads between tail of $c(s)$ and the endpoints. In
other words, $lk(L_s)$ is $\sum\epsilon(c_2)$ where the summation runs over all chords involved,
together with $c(s)$, into subdiagrams of the type $\raisebox{-10pt}{\includegraphics[width=.05\linewidth]{gausscode.png}}$.

We can also calculate $v_2$ by transforming $K$ to an ascending knotoid. By a similar argument we show that $v_2$ gets an expression similar to Eq. \ref{uknotoid}, where  for a crossing change at a chord $c(s)$, $lk(L_s)$ is $\sum\epsilon(c_2)$ where the summation runs over all chords involved,
together with $c(s)$, into subdiagrams of the type $\raisebox{-10pt}{\includegraphics[width=.05\linewidth]{gausscode2.png}}$.

Thus, we obtain $v_2(K)=\frac{1}{4}+6\hat{v}_2$.

\end{proof}

\section{The double alternating self-linking integral of curves in 3-space}\label{SLL}

In this section we define the double alternating linking integral of open curves in 3-space and examine its relation to the second Vassiliev measure of open curves in 3-space.

\begin{definition}\label{V2Geom}
Let $l$ denote a curve in 3-space with parametrization $\gamma$. We define the double alternating self-linking integral as:

\begin{equation}
    \begin{split}
        SLL(l)&=\frac{1}{8\pi}\int_{0}^1\int_{0}^{j_1}\int_{0}^{j_2}\int_{0}^{j_3}(\dot{\gamma}(j_1)\times\dot{\gamma}(j_3))\cdot\frac{\gamma(j_1)-\gamma(j_3)}{|\gamma(j_1)-\gamma(j_3)|^3}(\dot{\gamma}(j_2)\times\dot{\gamma}(j_4))\cdot\frac{\gamma(j_2)-\gamma(j_4)}{|\gamma(j_2)-\gamma(j_4)|^3}\\
        &\chi(j_1,j_2,j_3,j_4)dj_4dj_3dj_2dj_1,
    \end{split}
\end{equation}

\noindent where $\Gamma(s,t)=\frac{\gamma(s)-\gamma(t)}{|\gamma(s)-\gamma(t)|}$, for $s,t\in[0,1]$, $0\leq j_4<j_3<j_2<j_1\leq1$ and where $\chi(j_1,j_2,j_3,j_4)=1$, when $(j_1,j_2,j_3,j_4)\in E$ and $\chi(j_1,j_2,j_3,j_4)=0$, otherwise, where $E\subset[0,1]^4$, such that $0\leq j_1<j_2<j_3<j_4\leq 1$, $\Gamma(j_1,j_3)=-\Gamma(j_2,j_4)$ . 




\end{definition}

\begin{definition}
Let $j_1<j_2<j_3<j_4$ be points on a knot diagram. We will say that this 4-tuple corresponds to an alternating crossing when $j_1,j_3$ and $j_2,j_4$ are two crossing points in the diagram such that if $j_1$ belongs to the over-arc in the crossing, $j_2$ belongs to the under-arc in the crossing or vice-versa.  
\end{definition}

\begin{theorem}\label{V2Geom}
Let $l$ denote a curve in 3-space with parametrization $\gamma$. The double alternating self-linking integral can be expressed as:

\begin{equation}\label{sllg}
\begin{split}
    SLL(l)&=\frac{1}{8\pi}\int_{\vec{\xi}\in S^2}\sum_{j_1>j_2>j_3>j_4\in I_{\vec{\xi}}^*}\epsilon(j_1,j_3)\epsilon(j_2,j_4)dA,
\end{split}
\end{equation}

\noindent where $\epsilon(s,t)=\pm1$, is the sign of the crossing between the projection of $\gamma(s)$ and $\gamma(t)$, and where $I_{\vec{\xi}}^*$ denotes the set of 4-tuples of alternating crossings in the projection to the plane with normal vector $\vec{\xi}$.



\end{theorem}

\begin{proof}







Let $\gamma(t)$ denote a parametrization of $l$ and let $\Gamma(s,t)$ denote the Gauss map  $\Gamma(s,t)=\frac{\gamma(s)-\gamma(t)}{|\gamma(s)-\gamma(t)|}$. Let  $\gamma(j_1),\gamma(j_2),\gamma(j_3),\gamma(j_4)$ denote 4 points on $l$  such that $\Gamma(j_1,j_3)=-\Gamma(j_2,j_4)$. Then the projections of $\gamma(j_1),\gamma(j_3)$ and of $\gamma(j_2),\gamma(j_4)$ on the plane with normal vector $\vec{\xi}=\Gamma(j_1,j_3)$, coincide, creating two crossings, one where the arc containing $\gamma(j_1)$ is ``over'' and one where the arc containing $\gamma(j_2)$ is ``under''. Thus the crossings are ``alternating''. 

Let $l_n$ denote a polygonal approximation of $l$ obtained from a partition of the interval $[0,1]$. Let $\gamma(t),t\in[0,1]$ be a parametrization of the polygonal curve $l_n$. Then we can express the integral in the right hand-side of Eq. \ref{sllg} as

\begin{equation}
    \begin{split}
        &\frac{1}{8\pi}\int_{\vec{\xi}\in S^2}\sum_{j_1>j_2>j_3>j_4\in I_{\vec{\xi}}^*}\epsilon(j_1,j_3)\epsilon(j_2,j_4)dA\\
        &=\frac{1}{8\pi}\sum_{1\leq i\leq j \leq k\leq l\leq n}\int_{\vec{\xi}\in S^2}\epsilon^*(j_1,j_3)\epsilon^*(j_2,j_4)dA,
    \end{split}
\end{equation}

\noindent where $i,j,k,l$ are indices of the edges of the polygonal curve $l_n$ and where $\epsilon^*(s,t)$ can take values $0$, or $\pm1$, depending on whether the projection of $\gamma(s),\gamma(t)$ cross (and with what sign) or not. We note that the integral in the latter expression may be non zero, when either $i<j<k<l$ -thus, involving 4 edges- or when at most 2 of the edges are identified, ie. $i=j<k<l$ or $i<j=k<l$ or $i<j<k=l$. If more indices are identified, thus only two edges or only one are involved, it is impossible to have 2 crossings in their projection. We focus on the case where $i<j<k<l$, thus we have 4 different edges involved. (The case of 3 edges involved can be treated similarly and we will not discuss it in this proof). Two edges cross on a spherical quadrangle (and its antipodal) \cite{Banchoff1976}. Thus, two pairs of edges cross at the intersection of the two spherical quadrangles. Let us denote this intersection, which is a spherical polygon (and its antipodal) $A$. Then,

\begin{equation}
    \begin{split}
       &\frac{1}{8\pi}\int_{\vec{\xi}\in S^2}\epsilon^*(j_1,j_3)\epsilon^*(j_2,j_4)dA\\
       &=\frac{1}{8\pi}\int_{\vec{\xi}\in S^2\setminus A}\epsilon^*(j_1,j_3)\epsilon^*(j_2,j_4)dA+\frac{1}{8\pi}\int_{\vec{\xi}\in A}\epsilon^*(j_1,j_3)\epsilon^*(j_2,j_4)dA\\
       &=\frac{1}{8\pi}\int_{\vec{\xi}\in A}\epsilon(j_1,j_3)\epsilon(j_2,j_4)dA,
    \end{split}
\end{equation}

\noindent since $\int_{\vec{\xi}\in S^2\setminus A}\epsilon(j_1,j_3)\epsilon(j_2,j_4)dA=0$.
Let us denote $\gamma_1(j_1)$, $\gamma_2(j_2)$, $\gamma_3(j_3)$, $\gamma_4(j_4)$, where $j_1,j_2,j_3,j_4\in([0,1])^4$ the parametrizations of the edges $l,k,j,i$, respectively. Let us define the map $G^*$ from any 4-tuple in $[0,1]^4$ to $S^2\times S^2$, such that $G^*(j_1,j_2,j_3,j_4)=(\Gamma_1(j_1,j_3),\Gamma_2(j_2,j_4))$, where $\Gamma_1(j_1,j_3)=\frac{\gamma_1(j_1)-\gamma_3(j_3)}{|\gamma_1(j_1)-\gamma_3(j_3)|^3}$ and $\Gamma_2(j_2,j_4)=\frac{\gamma_1(j_2)-\gamma_3(j_4)}{|\gamma_1(j_2)-\gamma_3(j_4)|^3}$. Then by changing variables in the integral we obtain 

\begin{equation}
    \begin{split}
       &\frac{1}{8\pi}\int_{\vec{\xi}\in A}\epsilon(j_1,j_3)\epsilon(j_2,j_4)dA\\
       &=\frac{1}{8\pi}\int_{0^*}^{1^*}\int_{0^*}^{1^*}\int_{0^*}^{1^*}\int_{0^*}^{1^*}(\dot{\gamma_1}(j_1)\times\dot{\gamma_3}(j_3))\\
       &\cdot\frac{\gamma_1(j_1)-\gamma_3(j_3)}{|\gamma_1(j_1)-\gamma_3(j_3)|^3}(\dot{\gamma_2}(j_2)\times\dot{\gamma_4}(j_4))\cdot\frac{\gamma_2(j_2)-\gamma_4(j_4)}{|\gamma_2(j_2)-\gamma_4(j_4)|^3}dj_4dj_3dj_2dj_1,
    \end{split}
\end{equation}

\noindent where $|(\dot{\gamma(j_1)}\times\dot{\gamma(j_3)})\cdot\frac{\gamma(j_1)-\gamma(j_3)}{|\gamma(j_1)-\gamma(j_3)|^3}(\dot{\gamma(j_2)}\times\dot{\gamma(j_4)})\cdot\frac{\gamma(j_2)-\gamma(j_4)}{|\gamma(j_2)-\gamma(j_4)|^3}|$ is the Jacobian of $\Gamma^*$ and where the sign of  $(\dot{\gamma(j_1)}\times\dot{\gamma(j_3)})\cdot\frac{\gamma(j_1)-\gamma(j_3)}{|\gamma(j_1)-\gamma(j_3)|^3}$ is the sign of the crossing of the projections of the edges $e_1,e_3$ (when they cross in a projection direction) and $(\dot{\gamma(j_2)}\times\dot{\gamma(j_4)})\cdot\frac{\gamma(j_2)-\gamma(j_4)}{|\gamma(j_2)-\gamma(j_4)|^3}$ is the sign of the crossing of the projections of the edges $e_2,e_4$ (when they cross in a projection direction). The symbol $*$ in the integral indicates integration over the subset of $[0,1]^4$ which defines 4-tuples of points on the knot that define vectors which give alternating crossings. This is the subset of $[0,1]^4$ whose image through $\Gamma^*$ is $A$, we denote this subset $E=([0,1]^4)^*$. Since $\Gamma$ is a continuous function and $A$ is measurable, its pre-image, $E=([0,1]^4)^*$, is also measurable.  Instead of integrating over $E$ we can integrate over $[0,1]^4$ as follows


\begin{equation}
    \begin{split}
       &\frac{1}{8\pi}\int_{\vec{\xi}\in A}\epsilon(j_1,j_3)\epsilon(j_2,j_4)dA\\
       &=\frac{1}{8\pi}\int_0^1\int_0^1\int_0^1\int_0^1(\dot{\gamma_1}(j_1)\times\dot{\gamma_3}(j_3))\cdot\frac{\gamma_1(j_1)-\gamma_3(j_3)}{|\gamma_1(j_1)-\gamma_3(j_3)|^3}(\dot{\gamma_2}(j_2)\times\dot{\gamma_4}(j_4))\cdot\frac{\gamma_2(j_2)-\gamma_4(j_4)}{|\gamma_2(j_2)-\gamma_4(j_4)|^3}\\
       &\chi(j_1,j_2,j_3,j_4)dj_4dj_3dj_2dj_1,
    \end{split}
\end{equation}

\noindent where $\chi(j_1,j_2,j_3,j_4)=1$, when $(j_1,j_2,j_3,j_4)\in E$ and $\chi(j_1,j_2,j_3,j_4)=0$, otherwise. 

At the limit when $n\rightarrow\infty$, we obtain the formula:

\begin{equation}
\begin{split}
    &\frac{1}{8\pi}\int_{\vec{\xi}\in S^2}\sum_{j_1>j_2>j_3>j_4\in I_{\vec{\xi}}}\epsilon(j_1,j_3)\epsilon(j_2,j_4)dA\\
    &=\frac{1}{8\pi}\int_{j_1}\int_{j_2}\int_{j_3}\int_{j_4}(\dot{\gamma}(j_1)\times\dot{\gamma}(j_3))\cdot\frac{\gamma(j_1)-\gamma(j_3)}{|\gamma(j_1)-\gamma(j_3)|^3}(\dot{\gamma}(j_2)\times\dot{\gamma}(j_4))\cdot\frac{\gamma(j_2)-\gamma(j_4)}{|\gamma(j_2)-\gamma(j_4)|^3}\\
    &\chi(j_1,j_2,j_3,j_4)dj_4dj_3dj_2dj_1\\
    &=SLL(l).
\end{split}
\end{equation}

\end{proof}

\begin{proposition}\label{w2prob}

Let $l$ denote a closed curve in 3-space. Then the double alternating self-linking integral is a topological invariant and it is related to the second Vassililev invariant of the enhanced Jones polynomial of $l$ by

\begin{equation}\label{openv2}
v_2(l)=\frac{1}{4}+6SLL(l).
\end{equation}


\end{proposition}

\begin{proof}
By Theorem \ref{V2Geom} and by Theorem \ref{knotoid2},

\begin{equation}
v_2(l)=\frac{1}{4}+6\Bigl(\frac{1}{2}\sum_{j_1>j_2>j_3>j_4\in I_{\vec{\xi}}^*}\epsilon(j_1,j_3)\epsilon(j_2,j_4)\Bigr),
\end{equation}

\noindent where $I_{\vec{\xi}}^*$ denotes the set of pairs of alternating crossings in a diagram $l_{\vec{\xi}}$.

Since $v_2(l)$ is an invariant, it is independent of the projection direction.
We can also express $v_2(l)$ it as follows:

\begin{equation}\label{openv2}
\begin{split}
    v_2(l)&=\frac{1}{4\pi}\int_{\vec{\xi}\in S^2}\frac{1}{4}+6\Bigl(\frac{1}{2}\sum_{j_1>j_2>j_3>j_4\in I_{\vec{\xi}}^*}\epsilon(j_1,j_3)\epsilon(j_2,j_4)\Bigr)dA\\
    &=\frac{1}{4}+6\frac{1}{8\pi}\int_{\vec{\xi}\in S^2}\sum_{j_1>j_2>j_3>j_4\in I_{\vec{\xi}}^*}\epsilon(j_1,j_3)\epsilon(j_2,j_4)\Bigr)dA.
\end{split}
\end{equation}

By Theorem \ref{V2Geom}, $v_2(l)=\frac{1}{4}+6SLL(l)$.

\end{proof}

\begin{remark}
Note that this method of writing an integral in space for a Vassiliev invariant would work for any combinatorial expression for a Vassiliev invariant. The double alternating self-linking integral has similarities with the second Vassiliev invariant integral expression obtained from the perturbative expansion of Witten's integral \cite{Witten1989}. This relation will be explored in a sequel to this study.
\end{remark}

\begin{proposition}\label{cont}
Let $l$ denote an open curve in 3-space. $SLL(l)$ is a continuous function of the coordinates of $l$.
\end{proposition}

\begin{proof}
We will prove this first for a polygonal curve of $n$ edges and the result will follow for any curve $l\in R^3$ as $n\rightarrow\infty$.

The double alternating self-linking integral of a polygonal curve can be expressed as

\begin{equation}\label{finp}
    \begin{split}
        SLL(l_n)&=\frac{1}{8\pi}\int_{\vec{\xi}\in S^2}\sum_{j_1>j_2>j_3>j_4\in I_{\vec{\xi}}^*}\epsilon(j_1,j_3)\epsilon(j_2,j_4)dA\\
        &=\frac{1}{8\pi}\sum_{1\leq i\leq j \leq k\leq l\leq n}\int_{\vec{\xi}\in S^2}\epsilon(j_1,j_3)\epsilon(j_2,j_4)dA\\
        &=\frac{1}{2}\sum_{1\leq i\leq j \leq k\leq l\leq n}p_{j_1,j_2,j_3,j_4}\epsilon(j_1,j_3)\epsilon(j_2,j_4),\\
    \end{split}
\end{equation}


\noindent where $p_{j_1,j_2,j_3,j_4}$ denotes the geometric probability that the edges $e_{j_1},e_{j_3}$ and $e_{j_2},e_{j_4}$ both cross in a projection direction and give an alternating crossing.

We can express $p_{j_1,j_2,j_3,j_4}$ as the joint probability that $e_{j_1},e_{j_3}$ and $e_{j_2},e_{j_4}$ both cross. In \cite{Panagiotou2020b}, it was proved that the geometric probability that $e_{j_1},e_{j_3}$ cross, $p_{j_1,j_3}$, and the geometric probability that $e_{j_2},e_{j_4}$ cross, $p_{j_2,j_4}$ are continuous and are equal to the areas of the corresponding quadrangles on the sphere. Their intersection, $p_{j_1,j_2,j_3,j_4}$, is the area of the intersection of the two spherical quadrangles. Since both areas are continuous functions of the coordinates of the involved edges, so does their intersection.

\end{proof}

\begin{corollary}
Let $l$ denote an open curve in 3-space.
Then, as the endpoints of $l$ tend to coincide, $SLL(l)$ tends to $\frac{1}{6}v_2(l)-\frac{1}{24}$.
\end{corollary}

\begin{proof}
It follows directly from Proposition \ref{w2prob} and Proposition \ref{cont}.
\end{proof}

\begin{definition}
We define a tight open knot to be a fixed open curve in 3-space whose projections give only knot-type knotoids. 
\end{definition}

\begin{proposition}
Suppose that $l$ is a tight open knot. Then the double alternating self-linking integral is related to the second Vassiliev measure of $l$ as $w_2(l)=\frac{1}{4}+6SLL(l)$.
\end{proposition}

\begin{proof}
Let $l$ denote a tight knot in 3-space. The second Vassiliev measure is defined as follows:

\begin{equation}\label{v2av}
    w_2(l)=\frac{1}{4\pi}\int_{\vec{\xi}\in S^2}v_2(l_{\vec{\xi}})dA,
\end{equation}

\noindent where $v_2(l_{\vec{\xi}})$ is the second Vassiliev invariant of the knotoid that results from the projection of $l$ on the plane with normal vector $\vec{\xi}$.

Since $l$ is a tight knot, $l_{\vec{\xi}}$ is a knot-type knotoid for any $\vec{\xi}$ that defines a non-generic projection. 

The result follows by Theorem \ref{knotoid2} and Theorem \ref{V2Geom}.

\end{proof}

\section{The double alternating self-linking integral of a polygonal curve}\label{VasPol}

In the case of a polygonal curve, the double alternating self-linking integral has an expression as a finite sum of geometric probabilities.

\begin{proposition}
The double alternating self-linking integral of a polygonal curve (open or closed) can be expressed as follows:

\begin{equation}
\begin{split}
   SLL(l)&=\frac{1}{2}\sum_{1\leq j_4\leq j_3\leq j_2\leq j_1\leq n}p_{j_1,j_2,j_3,j_4}^*\epsilon(e_{j_1},e_{j_3})\epsilon(e_{j_2},e_{j_4})\\
    &+\frac{1}{2}\sum_{1\leq j_3\leq j_2\leq j_1\leq n}p_{j_1,j_2,j_3}^*\epsilon(e_{j_1},e_{j_2})\epsilon(e_{j_1},e_{j_3})\\
    \end{split}
\end{equation}

\noindent where  $p_{j_1,j_2,j_3,j_4}^*$ denotes the probability that the projections of $e_{j_1}$ and $e_{j_2}$ as well as the projections of $e_{j_3}$ and $e_{j_4}$ both cross in a projection and form an alternating crossing and $(p(j_1,j_2,j_3)^*$ denotes the probability that one of the edges $e_{j_1},e_{j_2},e_{j_3}$ intersects the other two in a projection and form an alternating crossing.

\end{proposition}

\begin{proof}

As explained in the proof of Proposition \ref{cont}, if $j_1,j_2,j_3,j_4$ all lie on the same edge, then there is no contribution to the integral. Similarly, if the 4 points lie only on 2 edges, they do not contribute to the integral any pairs of crossings. However, it is possible that they contribute when they lie in 3 or 4 edges. Similarly, if 3 edges or 2 pairs of edges cross in a projection, then they create two pairs of crossings. The result follows by separating these cases in Eq. \ref{finp}.





\end{proof}

This expression of $SLL$ shows that for polygonal curve its calculation relies in calculating the geometric probabilities $p_{j_1,j_2,j_3,j_4}^*$, $p_{j_1,j_2,j_3}^*$. We note that if all of $j_1,j_2,j_3,j_4$ are consecutive, then $p_{j_1,j_2,j_3,j_4}^*=0$. Similarly, if all of $j_1,j_2,j_3$ are consecutive, then $p_{j_1,j_2,j_3}^*=0$.
The next proposition provides a closed formula for the computation of $p_{j_1,j_2,j_3}^*$ in the case where two of $j_1,j_2,j_3$ are consecutive. 




\begin{corollary}
Let $e_i,e_j,e_{j+1}$ denote three edges in 3-space. Then the joint probability of crossing between the projections of $e_i,e_j$ and $e_i,e_{j+1}$ so that they give an alternating crossing, namely, $p(i,j,j+1)$ is equal to $\frac{1}{2\pi}A(Q_{i,j,j+1}^*)$, where $Q^*(i,j,j+1)$ is given in Table \ref{tableqij1}, where $c_{4,1}=(\vec{p}_{4,1}\cdot\vec{n}_1)\epsilon_{1,3}$,  $w=(u_2\times (-n_2))\cdot (u_2\times n_4)$, $w_0=(\vec{v}_3\times(-\vec{n}_1))\cdot(\vec{v}_3\times\vec{n}_3)$ and the vectors $\vec{u}_2,\vec{n}_2,\vec{n}_4,\vec{v}_3,\vec{v}_2$ and $\vec{n}_1$ are normal to the planes containing the vertices $(i-1,i,j+1),(i-1,i,j),(i-1,j-1,i),(j-1,j+1,j),(j-1,j+1,i)$, and $(i-1,j-1,j)$, respectively.



\end{corollary}

\begin{proof}
In order for them to form an alternating crossing, the signs of the two crossings must be the same. The geometric probability that 3 edges cross, 2 of which are consecutive, $e_{i},e_{j},e_{j+1}$ was found in Theorem A.1 in \cite{Panagiotou2020b}. In order to preserve the order of crossings, we need to ensure that as we move along $e_i$ in a projection, first we encounter the crossing with $e_j$ and then with $e_{j+1}$. This imposes the extra constraint that the spherical area is on the side of the great circle defined by the vector $\vec{v}_3$ in the direction of $\vec{v}_3$. The results are shown in Table \ref{tableqij1}.
\end{proof}

\begin{table}
\centering
\begin{tabular}{|l|l|}
\hline
$\epsilon_{i,j}=\epsilon_{i,j+1}, w<0, w_0<0$ & $Q_{i,j,j+1}^*$\\
\hline
$c_{j+1,i+1}>0, c_{j+2,i+1}>0, c_{j+1,i}>0, c_{j+2,i}>0$                   & $(\vec{n}_4,\vec{n}_1,-\vec{u}_2,\vec{v}_3)$\\
$c_{j+1,i+1}<0, c_{j+2,i+1}<0, c_{j+1,i}>0, c_{j+2,i}>0$                   & $(\vec{n}_4,-\vec{u}_3,-\vec{u}_2,\vec{v}_3)$\\
$c_{j+1,i+1}>0, c_{j+2,i+1}<0, c_{j+1,i}>0, c_{j+2,i}>0$                   & $(\vec{n}_4,\vec{n}_1,-\vec{u}_3,-\vec{u}_2,\vec{v}_3)$\\
$c_{j+1,i+1}<0, c_{j+2,i+1}>0, c_{j+1,i}>0, c_{j+2,i}>0$                   & $(\vec{n}_4,-\vec{u}_3,\vec{n}_1,-\vec{u}_2,\vec{v}_3)$\\
\hline
otherwise & 0\\
\hline
\end{tabular}
\caption{The spherical polygon $Q_{i,j,j+1}^*$ in the case where the signs satisfy $\epsilon_{i,j}=\epsilon_{i,j+1}$, depending on the conformation. The spherical polygon $Q_{i,j,j+1}^*$  contains the vectors which define planes where the projections of $e_i,e_j$ and $e_i,e_{j+1}$ both cross and they create an alternating crossing.  $(\vec{w}_1,\vec{w}_2,\dotsc,\vec{w}_n)$ denotes the spherical polygon bounded by the great circles with normal vectors $\vec{w}_i$,  $i=1,\dotsc, n$, in the counterclockwise orientation, (see \cite{Panagiotou2020b}).}
\label{tableqij1}
\end{table}





\subsection{The double alternating self-linking integral of a polygonal curve with 4 edges}

A polygonal curve with 4 edges is the shortest polygonal curve that can have a non-trivial double alternating self-linking integral. We will show that in this simple case the double alternating self-linking integral has an even simpler interpretation as the geometric probability that a projection of $l$ gives the knotoid $k2.1$.

\begin{figure}[H]
   \begin{center}
     \includegraphics[width=0.8\textwidth]{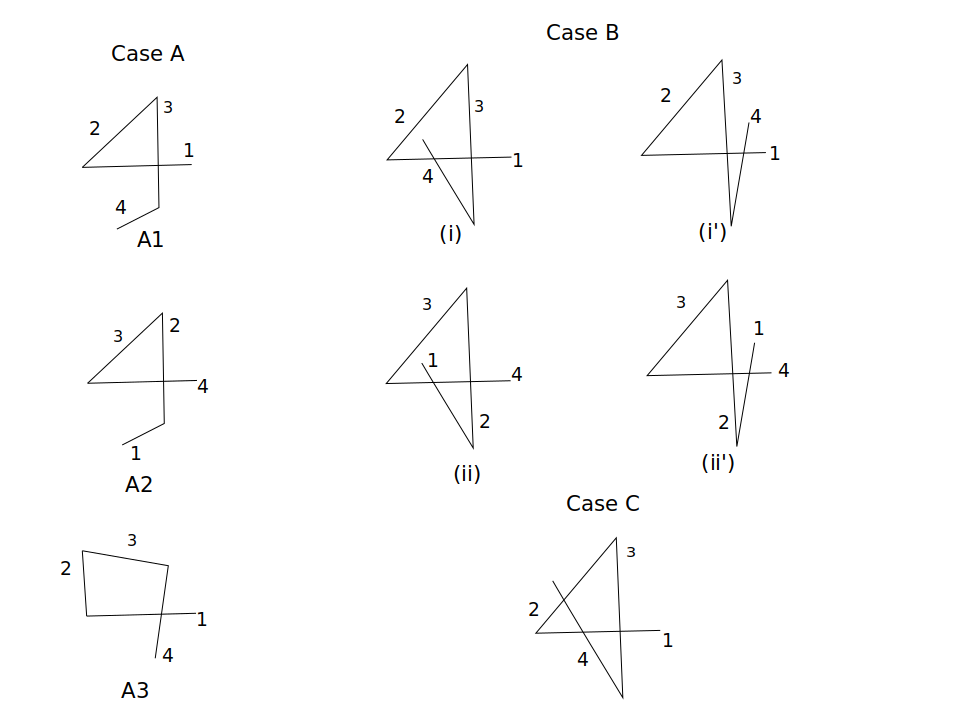}
     \caption{The possible projections with crossings of a polygonal curve with 4 edges.}
     \label{fig:cases}
   \end{center}
\end{figure}

\begin{proposition}
Let $l_4$ denote an open polygonal curve in 3-space. The double alternating self-linking integral of $l_4$ is equal to the signed geometric probability that it gives the knotoid k2.1 in a projection direction, ie. $SLL(l_4)=\frac{1}{2}P(l_{\vec{\xi}}=k2.1)$.
\end{proposition}

\begin{proof}
By Figure \ref{fig:cases} we see that the only possibility of a pair of alternating crossings in a projection of a curve with 4 edges is case B(i) or case B(ii), both of which correspond to the knotoid k2.1. In \cite{Panagiotou2020b} it was proved that for a given curve $l$ either B(i) or B(ii) is a possible outcome in the projections of $l$, but not both. The product of the two crossings is equal to 1 in both cases.

\end{proof}







\section{Conclusions}
In this manuscript we defined Vassiliev measures for open curves in 3-space and showed that they generalize the conventional Vassiliev invariants. For open curves, these are continuous functions of the curve coordinates which tend to the Vassiliev invariants of the closed curves as the endpoints tend to coincide. A geometric interpretation of Vassiliev measures of closed and open curves was given which allowed to derive a well defined integral expression of Vassiliev measures. More precisely, the double alternating self-linking integral was introduced and it was shown that it coincides with the second Vassiliev invariant of closed curves. For open curves, this integral is a continuous function of the curve coordinates and, when the open curves have tight knots, the double alternating self-linking integral coincides with the Vassiliev measure of the open curves. For polygonal curves, the double alternating self-linking integral has a simpler expression as a sum of finitely many geometric probabilities. The double alternating self-linking integral allows to rigorously define and capture entanglement and knotting in open curves in 3-space avoiding the calculation of polynomials and avoiding artificial closures. The method introduced in this work provides a framework in which other Vassiliev invariants can also be generalized to open curves. These measures provide a novel and efficient method to measure entanglement and knotting in physical systems of filaments and can lead to many impactful applications. 

\section{Acknowledgements}
Eleni Panagiotou was supported by NSF (Grant No. DMS-1913180). Louis Kauffman was supported by the 
Laboratory of Topology and Dynamics, 
Novosibirsk State University 
(contract no. 14.Y26.31.0025 
with the Ministry of Education and Science 
of the Russian Federation.)

\bibliographystyle{plain}

\bibliography{VassilievOpen}

\end{document}